\documentclass[10pt]{amsart}
\usepackage[pagebackref]{hyperref}
\usepackage{pinlabel}
\usepackage{overpic}
\usepackage{svg}
\usepackage{sseq}
\hypersetup{
    colorlinks = true,
    citecolor = [rgb]{0, 0.5, 0.00},
}

\renewcommand*{\backref}[1]{}
\renewcommand*{\backrefalt}[4]{%
    \ifcase #1 (Not cited.)%
    \or        (Cited on page~#2.)%
    \else      (Cited on pages~#2.)%
    \fi}

\usepackage{mathtools}
\usepackage{amssymb}

\usepackage[margin=1.175in]{geometry}
\usepackage{todonotes}
\usepackage{multirow}
\usepackage{longtable}
\usepackage{enumerate}
\usepackage{mathabx}
\usepackage{amsmath}
\usepackage{amsfonts}
\usepackage{amssymb}
\usepackage{float}
\usepackage{amsthm}
\usepackage{comment}
\usepackage{amscd}
\usepackage{amsxtra}
\usepackage{epsfig}
\usepackage{slashed}

\usepackage{listings}
\usepackage{color}

\definecolor{dkgreen}{rgb}{0,0.6,0}
\definecolor{gray}{rgb}{0.5,0.5,0.5}
\definecolor{mauve}{rgb}{0.58,0,0.82}

\usepackage{color}
\usepackage[all]{xy}
\usepackage{verbatim}

\usepackage{eucal}
\usepackage{mathrsfs}

\usepackage{graphicx}
\usepackage{tikz-cd}

\usepackage{url}
\usepackage{verbatim}
\usepackage{xy}
\xyoption{all}

\usepackage{amsthm}
\usepackage{tikz}
\usetikzlibrary{decorations.pathreplacing}

\usepackage{caption} 
\makeatletter
\def\@tocline#1#2#3#4#5#6#7{\relax
  \ifnum #1>\c@tocdepth % then omit
  \else
    \par \addpenalty\@secpenalty\addvspace{#2}%
    \begingroup \hyphenpenalty\@M
    \@ifempty{#4}{%
      \@tempdima\csname r@tocindent\number#1\endcsname\relax
    }{%
      \@tempdima#4\relax
    }%
    \parindent\z@ \leftskip#3\relax \advance\leftskip\@tempdima\relax
    \rightskip\@pnumwidth plus4em \parfillskip-\@pnumwidth
    #5\leavevmode\hskip-\@tempdima
      \ifcase #1
       \or\or \hskip 1em \or \hskip 2em \else \hskip 3em \fi%
      #6\nobreak\relax
    \hfill\hbox to\@pnumwidth{\@tocpagenum{#7}}\par% <---- \dotfill -> \hfill
    \nobreak
    \endgroup
  \fi}
\makeatother

\newtheorem{lemma}{Lemma}[section]

\newtheorem{theorem}[lemma]{Theorem}

\newtheorem{prop}[lemma]{Proposition}

\theoremstyle{definition}

\newtheorem{assumption}[lemma]{Assumption}
\newtheorem{definition}[lemma]{Definition}

\newtheorem{example}[lemma]{Example}

\numberwithin{equation}{section} % 在每个小节开始时重新编号方程
\newcommand{\Z}{\mathbb{Z}}

\newcommand{\CP}{\mathbf{C}P}

\newcommand{\R}{\mathbb{R}}

\newcommand{\D}{\slashed{\mathfrak D}}

\renewcommand{\S}{\mathbb{S}}
\renewcommand{\ss}{\mathfrak{s}}

\title{The Dehn twist on a connected sum of two homology tori}
\author{Haochen Qiu}

\begin{document}

\maketitle

\begin{abstract}
    Kronheimer-Mrowka showed that the Dehn twist along a $3$-sphere in the neck of the connected sum of two $K3$ surfaces is not smoothly isotopic to the identity. The tool they used is the nonequivariant family Bauer-Furuta invariant, and their result requires that the manifolds are simply connected and the signature of one of them is $16 \mod 32$. We generalize the $S^1$-equivariant family Bauer-Furuta invariant to nonsimply connected manifolds, and construct a refinement of this invariant. We use it to show that if $X_1,X_2$ are two homology tori such that their determinants $r_1,r_2$ are odd, then the Dehn twist along a $3$-sphere in the neck of $X_1\# X_2$ is not smoothly isotopic to the identity.
\end{abstract}

\tableofcontents
%\charpter{Introduction}
\section{Introduction}

%An exotic diffeomorphism is the a diffeomorphism that is continuously isotopic to the identity but 
%In this project we construct an exotic diffeomorphism on a nonsimply connected manifold without the need for an exotic smooth structure.

The goal of this work is to show that, the $S^1$-equivariant family Bauer-Furuta invariant detects nontrivial elements in the smooth mapping class group for a nonsimply connected manifold.

A homology $4$-torus is a smooth $4$-manifold that has the same homology groups as a $4$-dimensional torus $T^4$. %A connected sum of two closed $4$-manifolds $X_1$ and $X_2$, denoted by $X_1\# X_2$, is constructed by first removing a $4$-ball on each of them, then gluing them together along the boundaries of $X_1-D^4$ and $X_2-D^4$. 
The connected sum of two manifolds $X_1$ and $X_2$ can be written as 
\[
X_1\# X_2 = (X_1-D^4) \cup_{S^3} ([0,1]\times S^3)\cup_{S^3} (X_2-D^4),
\]
where $[0,1]\times S^3$ is called the neck of the connected sum. The Dehn twist along a $3$-sphere in the neck is a diffeomorphism $d: X_1 \# X_2 \to X_1 \# X_2$ such that $d$ is the identity outside the neck, and on the neck it has the form
\begin{align*}
[0,1]\times S^3 &\to [0,1]\times S^3 \\
(t, s ) &\mapsto (t,\alpha_t(s))
\end{align*}
where $\alpha\in  \pi_1(SO(4), Id) = \Z/2$ is the nontrivial element. It looks like you rotate your head by $2\pi$: Your head and body are in the original position, and the only part that changes is your neck. 

For a homology torus $X$, its cohomology groups are isomorphic to the ones of $T^4$, but the ring structure might be different. Let $\alpha_1,\cdots, \alpha_4$ be a basis of $H^1(X;\Z)$, and define the determinant of $X$ by 
\[
r := \left| \langle \alpha_1 \smile \alpha_2 \smile \alpha_3 \smile \alpha_4 , [X] \rangle\right|
\]
where $[X]$ is the fundamental class. The main theorem of this paper is 
\begin{theorem}\label{main-thm}
If $X1,X_2$ are two homology $4$-tori such that the determinants $r_1,r_2$ of them are odd. Then the Dehn twist along a $3$-sphere in the neck of $X_1\# X_2$ is not smoothly isotopic to the identity.
\end{theorem}

Remark that if $X1,X_2$ are standard tori, then the Dehn twist is even not homotopic to the identity, since they are aspherical (see \cite{McCULLOUGH81}). Hence Theorem \ref{main-thm} is only interesting when $X_1\#X_2$ is not aspherical.

The main tool we use is the Bauer-Furuta invariant \cite{BF02}. Its idea is to regard the Seiberg-Witten equation as a $\text{Pin}(2)$-equivariant map and to consider the property of the map. By a finite-dimensional approximation, it is an equivariant stable mapping class for a spin manifold $X$:
\[
\text{BF}^{\text{Pin}(2)} (X, \ss) \in \{ TF_0,  S^{n\mathbb{H}+b^+_2(X)\tilde{\R}} \}^{\text{Pin}(2)}.
\]
where $\ss$ is a $\text{Spin}^c$-structure of $X$, and $TF_0$ is the Thom space of a rank $m$ quarternion bundle over $Pic^{\ss}(X) = T^{b_1(X)}$, such that 
\[
m-n = -\frac{\sigma(X)}{4}
\]
where $\sigma(X)$ is the signature of $X$. 

One can also forget the $\text{Pin}(2)$-action and define the nonequivariant Bauer-Furuta invariant by 
\[
\text{BF}^{\{e\}} (X, \ss):= \text{Res}^{\text{Pin}(2)}_{\{e\}} \text{BF}^{\text{Pin}(2)} (X, \ss) \in \{ TF_0,  S^{4n+b^+_2(X)}\}.
\]
Now one has a sequence of invariants that can detect exotic phenomena: the Seiberg-Witten invariant, the nonequivariant Bauer-Furuta invariant, and the $\text{Pin}(2)$-equivariant Bauer-Furuta invariant. They contain more and more information, but the computations get more and more complicated. 

There are many previous results on the detections of the Bauer-Furuta invariant. Kronheimer-Mrowka~\cite{KM20} used the nonequivariant Bauer-Furuta invariant to show that the Dehn twist along the separating $3$-sphere in the neck of $K3\# K3$ is not smoothly isotopic to the identity. Lin~\cite{Lin23} used the $S^1$-equivariant Bauer-Furuta invariant to show that, the above Dehn twist on $K3\# K3$ is still not smoothly isotopic to the identity after a stabilization by $S^2 \times S^2$. These existing results are all for simply connected manifolds. Notice that the nonequivariant invariant does not work for Lin's result and the result in this paper, but the reasons are different.

Compared with the previous results, there are some issues when we consider homology tori:
\begin{enumerate}
\item[\textbullet] The Bauer-Furuta invariants of homology tori is unknown. The map from the Bauer-Furuta invariant to the Seiberg-Witten invariant is not well-defined because any homology torus doesn't satisfy the condition $b^+ - b_1 \geq 2$ mentioned in \cite{BF02}. Hence one cannot deduce the Bauer-Furuta invariant of a homology torus from its Seiberg-Witten invariant, just as what people did for the $K3$ surface.

\item[\textbullet] To produce a nontrivial family Bauer-Furuta invariant, Proposition 4.1 in \cite{KM20} requires that the signature of the manifold is $16 \mod 32$, and therefore it doesn't work for homology tori. The index of the twisted Dirac operator on any homology torus is zero, which leads to a vanishing twist of the quaternion bundles of a twisted spin family of homology tori. Therefore, the condition on the index of the Dirac operator in \cite{KM20} should be replaced (by a condition on the index bundle, as what we will see).

\item[\textbullet] The family Bauer-Furuta invariant of a homology torus is a stable mapping class from the index bundle to a sphere, which is hard to compute. The $K3$ surface considered in \cite{KM20} and \cite{Lin23}, however, has $b_1 = 0$. Hence the family Bauer-Furuta invariant of it is an element in the stable homotopy group of spheres, which is well known in low dimension ($1$, $2$ and $3$), and moreover, the equivariant version can be computed by algebraic topology. But for homology tori, the domain of their (family) Bauer-Furuta invariants is a possibly nontrivial bundle.
\end{enumerate}

The main theorem of this paper comes from a sequence of results:

First, by a perturbation of the SW equation proposed by Ruberman-Strle~\cite{RS00}, and a computation of the bundle $TF_0$ via the index theorem and the Steenrod square, we get 
\begin{theorem}
If $X$ is a homology torus with odd determinant, and $\ss$ is the trivial structure, then
\[
BF^{\{e\}} (X, \ss) = (0,0,0,0,1) \in 4\Z\oplus\Z/2.
\]
Actually $BF^{\{e\}} (X, \ss)$ is the Hopf element $\eta$.
\end{theorem}

Second, we compute the nonequivariant family Bauer-Furuta invariant for the mapping torus of the Dehn twist $d: X_1 \# X_2 \to X_1 \# X_2$. It is denoted by
\[
FBF^{\{e\}} ((X_1 \times S^1, \tilde{\ss}_1) \# (X_2 \times S^1, \tilde{\ss}_2^\tau)) \in \{S^{\R}\wedge TF , S^{2\mathbb{H} + 6\R} \}.
\]
We compute the bundle $F$ by the index theorem, and prove that there exists a Hopf element $\nu$ in the stable CW structure of $TF$. Therefore, by Atiyah-Hirzebruch spectral sequence, 
\[
FBF^{\{e\}} ((X_1 \times S^1, \tilde{\ss}_1) \# (X_2 \times S^1, \tilde{\ss}_2^\tau)) 
\]
must be trivial. This vanishing result is similar to the fact that, a $3$-sphere can not be mapped to $\CP^2$ nontrivially, because the only way to construct a nontrivial map from the $3$-sphere to the $2$-skeleton of $\CP^2$ is to form the Hopf fibration $\eta: S^3 \to S^2$, but the $4$-cell in $\CP^2$ is attached to the $2$-cell by the Hopf element $\eta$.

The solution of this issue is to consider the equivariant mapping class. If $S^1$ acts on $S^3$ by the Hopf fibration and acts on $\CP^2$ trivially, then we have
\begin{align*}
\{ S^3, \CP^2\}^{S^1} &\cong \{ S^3/S^1, \CP^2\} \\
&= \{ S^2, \CP^2\}\\
&\cong \Z
\end{align*}
since the group action is free in the domain (see \cite{Adams} Theorem 5.3). We hope to apply the same trick to the equivariant family Bauer-Furuta invariant $FBF^{S^1}$. However, the $S^1$-action on the domain of this invariant is not free away from the base point.

Therefore, we define a refinement of such invariant, which we call it the quotient invariant $FBF^{S^1}_{quotient}$. Finally, we compute the quotient $S^1$-equivariant family Bauer-Furuta invariant 
\[
BF^{S^1}_{quotient} ((X_1 \times S^1, \tilde{\ss}_1) \# (X_2 \times S^1, \tilde{\ss}_2^\tau)).
\]
By a cofiber sequence we can throw away the fixed points in the equivariant map, and then apply the equivariant Hopf theorem to convert $BF^{S^1}_{quotient} ((X_1 \times S^1, \tilde{\ss}_1) \# (X_2 \times S^1, \tilde{\ss}_2^\tau))$ to a nonequivariant stable mapping class. Now the dimension is changed and the Hopf invariant $\nu$ mentioned above has no effect. Hence we can apply the method of Kronheimer-Mrowka~\cite{KM20}, and conclude that
\begin{theorem}
$BF^{S^1}_{quotient} ((X_1 \times S^1, \tilde{\ss}_1) \# (X_2 \times S^1, \tilde{\ss}_2^\tau))$ is nontrivial.
\end{theorem}

\subsection*{Acknowledgements}
The author wants thank his advisor Daniel Ruberman for asking the main problem of this work and for his suggestions on the references. The author wants to express gratitude to Jianfeng Lin for a lot of valuable discussions. The author wants to thank Hokuto Konno for his explanation on the condition of globally defined $S^1$-equivariant families Bauer-Furuta invariant for a family of $4$-manifolds. The author wants to thank Yuwen Gu for explaining the method to compute the cohomotopy group of torus by Adam spectral sequence. This work is partially supported by NSF grant DMS-1952790.

\section{The family Bauer-Furata invariant for nonsimply connected manifolds}
In this section, we introduce the definition of the $\text{Pin}(2)$-equivariant Bauer-Furuta invariant for spin families of nonsimply connected manifolds, and a refinement of this invariant.

In \cite{BF02}, Bauer-Furuta introduce a finite dimensional approximation of the $\text{Pin}(2)$-equivariant Seiberg-Witten map. It is an equivariant stable mapping class for a spin manifold, and people call it the Bauer-Furuta invariant. In Baraglia-Konno~\cite{BK_2022_BF}, Kronheimer-Mrowka~\cite{KM20} and Lin~\cite{Lin23}, the family Bauer-Furuta invariant is introduced, but only for simply connected manifolds. For nonsimply connected manifolds, the original definition of Bauer-Furuta has to be generalized to define the family version.

Suppose $X$ is a closed spin $4$-manifold and $B$ is another closed manifold works as the parameter space. Suppose $E_X$ is a smooth family of $X$ over $B$, that is, a smooth bundle with fiber $X$ and base $B$. The fiber of $E_X$ over $b\in B$ would be denoted by $X_b$.

First we recall some definitions of spin structures (see Kronheimer-Mrowka~\cite{KM20} and Lin~\cite{Lin23}):

\begin{definition}
A family spin structure $\ss$ on $E_X$ is a double cover of the vertical frame bundle on $E_X$, suth that it restricts to the standard double cover $Spin(4) \to SO(4)$ on each fiber. A pair $(E_X, \ss)$ is called a spin family.
\end{definition}

\begin{definition}
A bundle isomorphism between two bundles $E_X$ and $E_X'$ over the same base is a diffeomorphism that preserves the fiber, and it restricts to a diffeomorphism on each fiber.
\end{definition}

\begin{definition}
Two spin families $(E_X,\ss)$ and $(E_X',\ss')$ are called isomorphic if there exists a bundle isomorphism $f_*: E_X \to E_X'$ such that $f(\ss)$ is isomorphic to $\ss'$.
\end{definition}

\begin{definition}
Let $\ss$ be a spin structure over $X$. When the base $B$ is a circle, there are two families spin structures over $E_X$ associated to $\ss$. The bundle $E_X$ is a mapping torus of some diffeomorphism $f: X\to X$, which induces a map on the frame bundle $f_*: \text{Fr}(X) \to \text{Fr}(X)$. There are two lifts of $f_*$ to the principle spin bundle $P(X)$ of $X$, and they differ from each other by a deck transformation $\tau: P(X) \to P(X)$. The mapping tori of these two lifts give rise to two families spin structures over $E_X$ associated to $\ss$. When $E_X$ is a product $B \times X$, they are called the \textbf{product spin structure} $\tilde{\ss}$ and the \textbf{twisted spin structure} $\tilde{\ss}^\tau$.
\end{definition}

\begin{example}
Let $(E_{X_1}, \ss_1)$ and $(E_{X_2}, \ss_2)$ be two spin families over circles. Let $\gamma_i$ be a section of $\gamma_i$. To form a family of $X_1 \# X_2$, one needs to choose an identification 
\[
\phi := \bigsqcup_{b\in B} (\phi_b: T_{\gamma_1(b)}(X_1)_b \stackrel{\cong}{\to} T_{\gamma_2(b)}(X_2)_b).
\]
The resulting bundle is denoted by $E_{X_1} \#_\phi E_{X_2}$, and it depends on $\phi$ up to homotopy. There is a unique $\phi$ (up to homotopy) such that $\ss_1$ and $\ss_2$ can be glued together to form a family spin structure on $E_{X_1} \#_\phi E_{X_2}$ (see Lin~\cite{Lin23} Lemma 2.25). The resulting spin family is denoted by $(E_{X_1}, \ss_1) \# (E_{X_2}, \ss_2)$.
\end{example}

Above definitions are the same as the treatment for simply connected manifolds (see Kronheimer-Mrowka~\cite{KM20} and Lin~\cite{Lin23}). In the rest of this section, we would discuss new assumptions and definitions we have to make for possible nonsimply connected manifolds.

We make the following assumptions through out the paper:
\begin{assumption} \label{assumption-holonomy-trivial}
The bundle $E_X$ satisfies:
\begin{enumerate}
\item \label{assumption-H1-trivial}
The action of $\pi_1(B,b)$ on $H^1(X_b,\Z)$ given by the holonomy of the bundle is trivial.
\item \label{assumption-H2-trivial}
The action of $\pi_1(B,b)$ on $H^2(X_b,\Z)$ given by the holonomy of the bundle is trivial.
\end{enumerate}
\end{assumption}
We work in the following settings:
\begin{enumerate}
\item[\textbullet] 
A family spin structure $\ss$ on $E_X$, which is a double cover of the vertical frame bundle on $E_X$, suth that it restricts to the double cover $Spin(4) \to SO(4)$ on each fiber.
\item[\textbullet] 
A family of metrics $g: B \to \text{Met}(X)$.
\item[\textbullet] 
Two quaternion bundles over $E_X$ given by $\ss$ and $g$:
\[
S^\pm := \bigsqcup_{b\in B}  S^\pm_b,
\]
where $S^\pm_b$ are positive and negative spinor bundles on $X_b$ given by $g_b$. Denote the space of spinors (sections of $S^\pm_b$) by $ \Gamma(S^\pm_b) $. Denote the parameterized Dirac operator over $B$ by
\[
\D_{A}(X_b) : \Gamma(S^+_b) \to   \Gamma(S^-_b)
\]
for $b \in B$ and $\text{spin}^c$-connection $A$ on $X_b$.

\item[\textbullet] 
A family of base points, which is a section $*: B \to E_X$.
\item[\textbullet] 
Define the action of the gauge group $\mathcal{G}(X_b) = \text{Map}(X_b, S^1)$ by letting $u\in \mathcal{G}(X_b)$ send $\Psi \in \Gamma(S^\pm_b) $ to $u\Psi \in \Gamma(S^\pm_b) $, and add $udu$ to the connection $1$-forms. Denote the based gauge group by 
\[
\mathcal{G}_0(*_b) := \{u \in \mathcal{G}(X_b)\mid u(*_b) = 1 \}.
\]
\item[\textbullet] Define
\begin{align*}
    \mathcal{B}^+{(\ss, g_b,*_b)} &:=( L^{k,2}(\mathscr{A}(\ss)) \oplus L^{k,2}(\Gamma(S_b^+))) /\mathcal{G}_0(*_b)\\
    \mathcal{B}^-{(\ss,g_b,*_b)} &:=(  L^{k-1,2}(i\Omega^0(X)/\R) \oplus  L^{k-1,2}(i\Omega^{+}(X)) \oplus L^{k-1,2}(\Gamma(S_b^-)))/\mathcal{G}_0(*_b),
\end{align*}
where $ \mathscr{A}(\ss)$ is the space of $U(1)$-connections of the determinant line bundle of the spin structure $\ss$. The ordinary Seiberg-Witten map is 
\begin{align}
    \label{equ:configOnG}
     \mathcal{F}_{(\ss, g_b,*_b)}: \mathcal{C}{(\ss, g_b,*_b)} &\to \mathcal{D}{(\ss, g_b,*_b)}\\
    \label{def:F}\mathcal{F}_{(\ss, g_b,*_b)}\begin{pmatrix}
         A\\
         \Phi\\
    \end{pmatrix}
    &= \begin{pmatrix}
         d^*(A-A_0)\\
         F^{+_{g_b}}_A  - \rho^{-1}(\sigma(\Phi, \Phi))\\
          \D_A \Phi\\
    \end{pmatrix}
\end{align}
where $\rho: \Omega^+(X) \to \mathfrak{su}(S^+_g)$ is the map defined by the Clifford multiplication, and $\sigma$
is the quadratic form given by $\sigma(\Phi, \Phi) = \Phi \otimes \Phi^* - \frac{1}{2} |\Phi|^2 id$, and $F^{+_{g_b}}_A$ is the self dual part of $F_A$ with respect to $g_b$. 
\end{enumerate}

To combine above infomation to a family Bauer-Furuta map, we need the cocycle condition to form the bundle 
\[
\bigsqcup_{b\in B} \mathcal{B}^\pm{(\ss, g_b,*_b)}
\]
in general. But when $B$ is $0$- or $1$-dimensional, we can pick an open cover of $B$ without three open sets overlapping. Hence we have
\begin{prop}
When $B$ is  a circle or a point, we have bundles
\[
  \mathcal{B}^\pm := \bigsqcup_{b\in B} \mathcal{B}^\pm{(\ss, g_b,*_b)}
\]
over $B$, and (\ref{def:F}) induces a bundle map $\mathcal{F}:\mathcal{B}^+ \to \mathcal{B}^-  $.
\end{prop}

Note that the map $\mathcal F$ is $\text{Pin}(2)$-equivariant, where $\text{Pin}(2)$ acts on $\Gamma(S_b^\pm)$ by the left quaternion multiplication, and acts on forms by reversing the sign.
Note also that over any $b\in B$, the fiber of $\mathcal{B}^+$ is a bundle over the Picard torus $Pic(X) = H^1(X;\R)/H^1(X;\Z)$. When $B$ contains only one point, we can apply the finite-dimensional approximation technique in \cite{BF02} on the map $\mathcal{F} : \mathcal{B}^+ \to \mathcal{B}^-  $. By the classical properness result (see \cite{BF02} Proposition 3.1), we can further convert the map between finite-dimensional vector spaces to a map between compact spaces by the one point compactification. Hence we have 
\begin{prop}[\cite{BF02} Corollary 3.2]
The map $\mathcal{F}$ defines an element 
\[
\text{BF}^{\text{Pin}(2)} (X, \ss)
\]
in the equivariant stable cohomotopy group
\[
 \{ TF_0,  S^{n\mathbb{H}+b^+_2(X)\tilde{\R}} \}^{\text{Pin}(2)},
\]
where $\ss$ is a $\text{Spin}^c$-structure of $X$, and $TF_0$ is the Thom space of a rank $m$ quarternion bundle over $Pic(X) = T^{b_1(X)}$, such that 
\[
m-n = -\frac{\sigma(X)}{4}.
\]
\end{prop}

Now consider the case where $B$ is a circle. By Assumption \ref{assumption-holonomy-trivial} (\ref{assumption-H2-trivial}), $\pi_1(B)$ acts trivially on $H^1(X;\R)$. Hence the torus bundle 
\[
 \bigsqcup_{b\in B} Pic(X_b)
\]
is trivial. By Assumption \ref{assumption-holonomy-trivial} (\ref{assumption-H2-trivial}), $\pi_1(B)$ acts trivially on $H^2(X;\R)$. Hence a homology orientation would give a trivialization of the bundle (as in \cite{KM20})
\[
\bigsqcup_{b\in B} H^2(X_b;\R).
\] 
Therefore the map $\mathcal{F}$ defines a $\text{Pin}(2)$-equivariant stable map
\[
\mu :  B \times TF \to B \times S^{n\mathbb{H}+b^+_2(X)\tilde{\R}}  \stackrel{pj}{\to} S^{n\mathbb{H}+b^+_2(X)\tilde{\R}}.
\]
By the definition of the Bauer-Furuta map, $\mu$ always sends the base point of $TF$ to the base point of $S^{n\mathbb{H}+b^+_2(X)\tilde{\R}}$. Hence $\mu$ has a lift 
\[
\tilde{\mu}: B_+ \wedge TF \to  S^{n\mathbb{H}+b^+_2(X)\tilde{\R}}
\]
such that the diagram 
% https://q.uiver.app/#q=WzAsMyxbMSwwLCJCIFxcdGltZXMgVEYiXSxbMCwxLCJCXysgXFx3ZWRnZSBURiA9IEIgXFx0aW1lcyBURiAvIEIgXFx0aW1lcyBcXHsqXFx9Il0sWzIsMSwiU157blxcbWF0aGJie0h9K2JeK18yKFgpXFx0aWxkZXtcXFJ9fSJdLFswLDEsInBqIiwyXSxbMCwyLCJcXG11Il0sWzEsMiwiXFx0aWxkZXtcXG11fSIsMCx7InN0eWxlIjp7ImJvZHkiOnsibmFtZSI6ImRhc2hlZCJ9fX1dXQ==
\[\begin{tikzcd}
	& {B \times TF} \\
	{B_+ \wedge TF = B \times TF / B \times \{*\}} && {S^{n\mathbb{H}+b^+_2(X)\tilde{\R}}}
	\arrow["pj"', from=1-2, to=2-1]
	\arrow["\mu", from=1-2, to=2-3]
	\arrow["{\tilde{\mu}}", dashed, from=2-1, to=2-3]
\end{tikzcd}\]
commutes.

To simplify the cell structure of the domain, we use the construction as what Lin used for simply connected case (see \cite{Lin23} Definition 2.21). Let $e: S^0 \to S^2$ be the embedding of $S^0$ to the south and the north pole of the sphere. The cone of $e$ would be denoted by $Ce \simeq S^2/S^0 $. One has a quotient map
\[
S^{2\R} \to Ce \simeq S^{\R} \wedge B_+.
\]
and therefore a quotient map 
\[
 S^{2\R}  \wedge TF \stackrel{q}{\longrightarrow} S^{\R} \wedge B_+ \wedge TF.
\]
We have
\begin{prop}
Under Assumption \ref{assumption-holonomy-trivial}, the stable map $\tilde{\mu}\circ q$ defines an element 
\[
\text{BF}^{\text{Pin}(2)} (E_X, \ss)
\]
in the equivariant stable cohomotopy group
\[
 \{ S^\R \wedge TF,  S^{n\mathbb{H}+b^+_2(X)\tilde{\R}} \}^{\text{Pin}(2)},
\]
where $\ss$ is a $\text{Spin}^c$-structure of $X$, and $TF$ is the Thom space of a rank $m$ quarternion bundle over $Pic^{\ss}(X) = T^{b_1(X)}$, such that 
\[
m-n = -\frac{\sigma(X)}{4}.
\]
\end{prop}

Now we consider a refinement of the family Bauer-Furuta invariant:
\begin{definition}\label{def:quotient_inv}
Denote the zero section of $F$ by $Pic(X)$. We also regard it as a subspace of $TF$. There exists a stable map $\bar \mu: S^{\R}\wedge (TF / Pic(X))\to S^{n\mathbb{H}+b^+_2(X)\tilde{\R}} $ such that the diagram  
% https://q.uiver.app/#q=WzAsMyxbMSwwLCJTXlxcUiBcXHdlZGdlIFRGIl0sWzAsMSwiU157XFxSfVxcd2VkZ2UgKFRGIC8gUGljKFgpKSJdLFsyLDEsIlNee25cXG1hdGhiYntIfStiXitfMihYKVxcdGlsZGV7XFxSfX0iXSxbMCwxLCJwaiIsMl0sWzAsMiwiXFx0aWxkZSBcXG11Il0sWzEsMiwiXFxiYXJ7XFxtdX0iLDAseyJzdHlsZSI6eyJib2R5Ijp7Im5hbWUiOiJkYXNoZWQifX19XV0=
\[\begin{tikzcd}
	& {S^\R \wedge TF} \\
	{S^{\R}\wedge (TF / Pic(X))} && {S^{n\mathbb{H}+b^+_2(X)\tilde{\R}}}
	\arrow["pj"', from=1-2, to=2-1]
	\arrow["{\tilde \mu}", from=1-2, to=2-3]
	\arrow["{\bar{\mu}}", dashed, from=2-1, to=2-3]
\end{tikzcd}\]
commutes. The stable map $\bar \mu$ defines an element:
\[
FBF^{\text{Pin}(2)}_\text{quotient} (E_X , \ss) \in \{S^{\R}\wedge (TF / Pic(X)) , S^{n\mathbb{H}+b^+_2(X)\tilde{\R}} \}^{\text{Pin}(2)},
\]
and we call it the quotient ${\text{Pin}(2)}$-equivariant family Bauer-Furuta invariant. Similarly, we define the quotient ${\text{Pin}(2)}$-equivariant Bauer-Furuta invariant to be an element
\[
\text{BF}_\text{quotient}^{\text{Pin}(2)} (X, \ss)
\]
in the equivariant stable cohomotopy group
\[
 \{ TF/Pic(X),  S^{n\mathbb{H}+b^+_2(X)\tilde{\R}} \}^{\text{Pin}(2)}.
\]
\end{definition}
The motivation of this refinement would be illustrated in the following sections.

This invariant works as well as the ordinary Bauer-Furuta invariant, because in the Seiberg-Witten equation (\ref{def:F}), the Picard torus with zero spinor is always sent to zero (while the kernel of the index bundle might be mapped to nonzero self-dual $2$-form). In particular, it detects $\pi_0(\text{Diff}(X))$ as what ordinary Bauer-Furuta invariant does:
\begin{prop}\label{prop:free-invariant-detect-pi-0-Diff}
If $f$ is smoothly isotopic to the identity, and $E_X$ is the mapping torus of $f$, then $\text{FBF}^{\text{Pin}(2)}_\text{quotient} (E_X , \ss)$ is trivial.
\end{prop}

\begin{prop}
Suppose $(E_{X_1}, \ss_1)$ is a spin family over a circle. Suppose $X_2$ is a spin $4$-manifold and $\ss_2$ is a spin structure of $X_2$. Consider the product spin family $(X_2 \times S^1, \tilde{\ss}_2)$. Then 
\begin{equation}\label{equ:gluing}
\text{FBF}^{\text{Pin}(2)}_\text{quotient} ((E_{X_1}, \ss_1) \#(X_2 \times S^1, \tilde{\ss}_2)) = \text{FBF}^{\text{Pin}(2)}_\text{quotient} (E_{X_1}, \ss_1) \wedge \text{BF}_\text{quotient}^{\text{Pin}(2)} (X_2, \ss_2)
\end{equation}
\end{prop}
\begin{proof}
In \cite{Bauer2}, Bauer proves a gluing theorem of the ordinay Bauer-Furuta invariant by performing a homotopy of the stable maps. In \cite{KM20} and \cite{Lin23}, such gluing theorem is extended to the families Bauer-Furuta invariant for a family of $4$-manifolds over a compact space. Their result is just Equation (\ref{equ:gluing}) without the subscript ``quotient''. The statement for the quotient invariant can be proved similarly, since the above homotopy always sends the zero section of the bunbdle $F$ (consists of flat connections with zero spinor) to the zero in the target.
\end{proof}

\section{The Bauer-Furuta invariant of homology tori}

\begin{theorem}
Suppose $X$ is a homology torus with 
\begin{equation}
\langle \alpha_1 \smile \alpha_2 \smile \alpha_3 \smile \alpha_4 , X \rangle = d,
\end{equation}
where $\{\alpha_i \}$ is a basis for $H^1(X;\Z)$, and $d$ is odd. Let $\ss$ be the $\text{spin}^c$ structure on $X$ with trivial determinent line. Then the nonequivariant Bauer-Furuta invariant $BF^{\{e\}} (X, \ss)$ is the generator of $\Z/2\Z$ in a group $4\Z \oplus \Z/2\Z$.
\end{theorem}

\begin{proof}
First we compute the group that the invariant lives. Let $F_0$ be the bundle $\mathbb{H} \to F_0 \to \mathcal{T}^4$ with $c_1(T_0) = 0$ and $c_2(T_0) = d$. Let $TF_0$ be its Thom space. The equivariant Bauer-Furuta invariant is 
\[
\text{BF}^{\text{Pin}(2)} (X, \ss) \in \{ TF_0,  S^{\mathbb{H}+3\tilde{\R}} \}^{\text{Pin}(2)}.
\]
Forget the $\text{Pin}(2)$-action then we get the non equivariant Bauer-Furuta invariant
\[
\text{BF}^{\{e\}} (X, \ss):= \text{Res}^{\text{Pin}(2)}_{\{e\}} \text{BF}^{\text{Pin}(2)} (X, \ss) \in \{ TF_0,  S^7 \}.
\]

A sketch of the proof: $TF_0$ can be regarded as a CW complex with one $0$-cell, one $4$-cell (corresponds to $S^\mathbb{H}$), four $5$-cells (obtained from the $1$-cells of $\mathcal{T}^4$ by multiplying a $4$-cell corresponds to $S^\mathbb{H}$), six $6$-cells (obtained from the $2$-cells of $\mathcal{T}^4$ similarly), four $7$-cells (obtained from the $3$-cells of $\mathcal{T}^4$), one $8$-cell (obtained from the $4$-cell of $\mathcal{T}^4$). By a cofiber sequence and the CW approximation theorem, there is an isomorphism
\[
\{ TF_0,  S^7 \} \cong \{ TF_0/TF_0^{(5)},  S^7 \}
\]
where $TF_0^{(5)}$ is the $5$-th skeleton of $TF_0$. We want to show that all attaching maps of $TF_0/TF_0^{(5)}$ are trivial.

First by Thom isomorphism and the cohomology of $\mathcal{T}^4$, we deduce that all cells of $TF_0/TF_0^{(5)}$ survive in the cohomology group, so all adjacent attaching maps are trivial.

Now the only possible nontrivial attaching map is the one from the $8$-cell to $6$-cells. Since the only nontrival element of $\pi_1$ is the Hopf map $\eta$, which can be detected by the Steenrod square, it suffices to show that $Sq^2$ is trivial. Let $u$ be the Thom class of $F_0$. Note that $u$ is an element in $H^*(F_0, F_0- \mathcal{T}^4)$ represented by the zero section of $F_0$. The cup product of $H^*(F_0, F_0- \mathcal{T}^4)$ with $H^*(\mathcal{T}^4)$ still produces closed submanfolds, so $H^*(\mathcal{T}^4)$ acts on $H^*(F_0, F_0- \mathcal{T}^4)$. By Cartan formula we have for any $x \in H^*(\mathcal{T}^4)$
\begin{align*}
Sq^n(ux) &= \sum_{i+j = n} Sq^i (u) Sq^j(x) \\
&=  \sum_{i+j = n}u w_i(F_0) Sq^j(x)
\end{align*}
where $w_i$ is the $i$-th Stiefel–Whitney class. The cohomology of $\mathcal{T}^4$ is an exterior algebra generated by four $1$-classes which the Steenrod algebra acts on trivially. By the Cartan formula and the Adem relations, the Steenrod algebra acts on $H^*(\mathcal{T}^4)$ trivially. Hence $Sq^j(x) \neq 0$ iff $j=0$. So 
\[
Sq^2(ux) = uw_2(F_0)Sq^0(x).
\]
But $w_2(F_0) \equiv c_1(F_0) \mod 2$ and from the structure of $F_0$ we have $c_1(F_0) = 0$. So the attaching maps from the $8$-cell to $6$-cells are trivial.

Now we conclude that $TF_0/TF_0^{(5)}$ is equivalent to $6\S^6 \vee 4\S^7 \vee \S^8 $ in the stable category. Hence 
\[
\{ TF_0,  S^7 \} \cong \{ TF_0/TF_0^{(5)},  S^7 \}\cong [ 6\S^6 \vee 4\S^7\vee \S^8,  \S^7 ] \cong 4\Z\oplus\Z/2.
\]

By Ruberman-Strle~\cite{RS00}, the preimage of a generic point under $BF^{\{e\}} (X, \ss)$ is a Lie framed circle in a fiber of $TF_0$. So the restriction $BF^{\{e\}} (X, \ss)|_{\S^8}$ is a suspension of the Hopf map
\[
\Sigma^5 \eta: \S^8 \to \S^7.
\]
And the restrictions of $BF^{\{e\}} (X, \ss)$ on those $7$-cells have degree $0$, otherwise the preimage of a generic point would contain discrete points. Therefore,
\[
BF^{\{e\}} (X, \ss) = (0,0,0,0,1) \in 4\Z\oplus\Z/2.
\]

\end{proof}

\section{Family Bauer-Furuta invariant of the connected sum of two tori}

Suppose $X_1$ and $X_2$ are homology tori with determinant $r_1$ and $r_2$. Let $X = X_1 \# X_2$ and $\mathcal{T}^8 \cong H^1(X, S^1)$. Let $\ss$ be the $\text{spin}^c$ structure on $X$ with trivial determinant line. Denote the family of Dirac operators by $\D$. The index bundle $Ind({\D})$ is an $\mathbb{H}$-bundle over $\mathcal{T}^8 $ since $X$ is spin. Let $\mathbf{L} \to X \times \mathcal{T}^8 $ be the universal line bundle. Since the $\hat{A}$-genus of $X$ is zero, the Chern character of the index bundle is
\[
\text{ch}(Ind(\D)) =\int_X \text{ch}(\mathbf{L})
\]
by Atiyah-Singer. 

Suppose $\mathbf{L}$ is equipped with a connection 
\[ 
\mathbf{A} = 2\pi i \sum_{k=1}^{8} t_k \alpha_k
\]
where $\alpha_1, \dots, \alpha_4$ is a basis for $H^1(X_1,\Z)$ and $\alpha_5, \dots, \alpha_8$ is a basis for $H^1(X_2,\Z)$, and $t_k \mapsto 2\pi i t_k \alpha_k$ are coordinates on $\mathcal{T}^8 \cong H^1(X, i\R) /H^1(X, 2\pi i\Z)  $. Then the first Chern class of $\mathbf{L}$ is 
\[
\Omega = \sum_{k} \alpha_k\wedge dt_k.
\]
By the dimension reason 
\[
\text{ch}(\mathbf{L}) = 1 + \Omega + \frac12 \Omega^2 + \frac16 \Omega^3+ \frac{1}{24} \Omega^4.
\]
A term in $\text{ch}(\mathbf{L})$ contains a volume form of $X$ if and only if it contains $\alpha_1 \smile \alpha_2 \smile \alpha_3 \smile \alpha_4 $ or $\alpha_5 \smile \alpha_6 \smile \alpha_7 \smile \alpha_8 $. Therefore,
\[
\text{ch}(Ind(\D)) = \pm r_1 [vol_{\mathcal{T}^4_1}]\pm r_2 [vol_{\mathcal{T}^4_2}]
\]
where $\mathcal{T}^4_i \cong  H^1(X_i, S^1)$ are submanifolds of $\mathcal{T}^8$. Hence $c_2(Ind(\D)) = \pm r_1 [vol_{\mathcal{T}^4_1}]\pm r_2 [vol_{\mathcal{T}^4_2}]$ and $c_i(Ind(\D) ) = 0$ for $i\neq 2$. Hence
\[
\text{BF}^{\{e\}} (X, \ss) \in \{ TF,  S^{\mathbb{H} + 6\R} \}
\]
where $TF$ is the Thom space of the $\mathbb{H}$-bundle $\mathbb{H} \to F \to \mathcal{T}^8 $. Hence we have

\begin{lemma}
$BF^{\{e\}} ((X_1 \times S^1, \tilde{\ss}_1) \# (X_2 \times S^1, \tilde{\ss}_2^\tau)) \in \{S^{\R}\wedge TF , S^{\mathbb{H} + 6\R} \} = \{S^{\R}\wedge TF , S^{10} \}$
\end{lemma}

This also follows from the gluing theorem of Bauer-Furuta invariant, which asserts that the domain of the stable cohomotopy element for a connected sum is an external product of the domains of two elements.

\begin{theorem}
Suppose $X_1$ and $X_2$ are homology tori with odd determinant. Let $\ss_i$ be the $\text{spin}^c$ structure on $X_i$ with trivial determinant line bundle. Then the nonequivariant family Bauer-Furuta invariant $BF^{\{e\}} ((X_1 \times S^1, \tilde{\ss}_1) \# (X_2 \times S^1, \tilde{\ss}_2^\tau))\in \{S^{\R}\wedge TF , S^{10} \}$ is trivial.
\end{theorem}

\begin{proof}
First we compute the structure of $TF$: 
$TF$ can be regarded as a CW complex with one $0$-cell, one $4$-cell (corresponds to $S^\mathbb{H}$), several $n$-cells ($n=5,6,7,8,9,10,11$, obtained from the $(n-4)$-cells of $\mathcal{T}^8$ by multiplying the $4$-cell corresponds to $S^\mathbb{H}$), and one $12$-cell (obtained from the $8$-cell of $\mathcal{T}^8$).

We claim that the attaching maps from the $12$-cell to two of the $8$-cells are Hopf element $\nu$.  Let $u$ be the Thom class of $F$. Note that $u$ is an element in $H^*(F, F- \mathcal{T}^8)$ represented by the zero section of $F$. The cup product of $H^*(F, F- \mathcal{T}^8)$ with $H^*(\mathcal{T}^8)$ still produces closed submanfolds, so $H^*(\mathcal{T}^8)$ acts on $H^*(F, F- \mathcal{T}^8)$. By Cartan formula we have for any $x \in H^*(\mathcal{T}^8)$
\begin{align*}
Sq^n(ux) &= \sum_{i+j = n} Sq^i (u) Sq^j(x) \\
&=  \sum_{i+j = n}u w_i(F_0) Sq^j(x)
\end{align*}
where $w_i$ is the $i$-th Stiefel–Whitney class. The cohomology of $\mathcal{T}^8$ is an exterior algebra generated by four $1$-classes which the Steenrod algebra acts on trivially. By the Cartan formula and the Adem relations, the Steenrod algebra acts on $H^*(\mathcal{T}^8)$ trivially. Hence $Sq^j(x) \neq 0$ iff $j=0$. So 
\[
Sq^4(ux) = uw_4(F)Sq^0(x).
\]
But $w_4(F) \equiv c_2(F) \mod 2$ and from the structure of $F$ we have $c_2(F) = \pm r_1 [vol_{\mathcal{T}^4_1}]\pm r_2 [vol_{\mathcal{T}^4_2}]$.  Note that $[vol_{\mathcal{T}^4_1}]\cup  [vol_{\mathcal{T}^4_2}]=[vol_{\mathcal{T}^8}]$ and $[vol_{\mathcal{T}^4_i}]\cup  [vol_{\mathcal{T}^4_i}]=0$. Hence if $\{i,j\}=\{1,2\}$ and $r_i$ is odd, 
\[
Sq^4(u [vol_{\mathcal{T}^4_j}])=  u  [vol_{\mathcal{T}^4_i}]\cup  [vol_{\mathcal{T}^4_j}] = u [vol_{\mathcal{T}^8}]
\]
is dual to the $12$-cell. Since the Hopf element $\nu$ is detected by $Sq^4$, the attaching map from the $12$-cell (dual to $u [vol_{\mathcal{T}^8}]$) to a $8$-cell (dual to $u [vol_{\mathcal{T}^4_j}]$) is $\nu$.

Therefore, in $S^{\R}\wedge TF$, the attaching map from the $13$-cell to a $9$-cell is $\nu$. Denote the $13$-cell by $\sigma$ and the $9$-cell by $\tau$. Note that $\nu$ is the generator of $\pi^{-3}(*) = \pi_3(\S) = \Z/24$. We can regard $\sigma$ be the generator of 
\[{\pi^{10}(\Sigma TF^{(13)}/\Sigma TF^{(12)})} = {C^{13}(\Sigma TF;\pi^{-3}(*))},
\]
while $\tau$ is a generator of 
\[{\pi^{10}(\Sigma TF^{(9)}/\Sigma TF^{(8)})} = {C^{9}(\Sigma TF;\pi^{0}(*))},
\]
where $C^*$ is the cochain group. Consider the $E_1$ page of the Atiyah-Hirzebruch spectral sequence:
% https://q.uiver.app/#q=WzAsNyxbMywwLCJIXnsxM30oXFxTaWdtYSBURl57KDEzKX0vXFxTaWdtYSBURl57KDEyKX07XFxwaV4wKCopKSJdLFszLDEsIkheezEzfShcXFNpZ21hIFRGXnsoMTMpfS9cXFNpZ21hIFRGXnsoMTIpfTtcXHBpXnstMX0oKikpIl0sWzMsMiwiSF57MTN9KFxcU2lnbWEgVEZeeygxMyl9L1xcU2lnbWEgVEZeeygxMil9O1xccGleey0yfSgqKSkiXSxbMywzLCJIXnsxM30oXFxTaWdtYSBURl57KDEzKX0vXFxTaWdtYSBURl57KDEyKX07XFxwaV57LTN9KCopKSJdLFswLDBdLFsxLDAsIkheezl9KFxcU2lnbWEgVEZeeyg5KX0vXFxTaWdtYSBURl57KDgpfTtcXHBpXjAoKikpIl0sWzIsMCwiXFxjZG90cyJdLFszLDUsImReNCIsMCx7ImN1cnZlIjotNX1dXQ==
\[\begin{tikzcd}
	{C^{8}(\Sigma TF;\pi^0(*))}& {C^{9}(\Sigma TF;\pi^0(*))} & \cdots & {C^{13}(\Sigma TF;\pi^0(*))} \\
	&&& {C^{13}(\Sigma TF;\pi^{-1}(*))} \\
	&&& {C^{13}(\Sigma TF;\pi^{-2}(*))} \\
	&&& {C^{13}(\Sigma TF;\pi^{-3}(*))}
	\arrow["{d_1}",  from=1-1, to=1-2]
	%\arrow["{d_4}", dashed, bend right=23, from=1-2, to=4-4]
\end{tikzcd}\]
$\tau$ comes from an $8$-cell of $TF$, and by the homology of the torus and the Thom isomorphism theorem, it is nontrivial in the homology group. Hence $\tau$ survives to the $E_2$-page:
\[\begin{tikzcd}
	{} & {H^{9}(\Sigma TF;\pi^0(*))} & \cdots & {H^{13}(\Sigma TF;\pi^0(*))} \\
	&&& {H^{13}(\Sigma TF;\pi^{-1}(*))} \\
	&&& {H^{13}(\Sigma TF;\pi^{-2}(*))} \\
	&&& {H^{13}(\Sigma TF;\pi^{-3}(*))}
	\arrow["{d_4}", dashed, bend right=23, from=1-2, to=4-4]
\end{tikzcd}\]
Since all arrows point down to the right after the $E_2$-page, $\tau$ can not be hit by $d_i$ for $i>1$. Hence $\tau$ survives to the $E_4$-page. If $\sigma$ survives to the $E_4$-page, It must be $d_4(\tau)$, since the attaching map from $\sigma$ to $\tau$ is $\nu$ and $\nu$ is the generator of $\pi^{-3}(*) = \pi_3(\S) = \Z/24$. Hence $\sigma$ and therefore any element in ${H^{13}(\Sigma TF;\pi^{-3}(*))}$ doesn't survive to the $E_\infty$ page. So the $13$-cell can only be mapped to $S^{10}$ trivially. 

From the observation of Kronheimer-Mrowka~\cite{KM20}, the preimage of a generic point under $BF^{\{e\}} ((X_1 \times S^1, \tilde{\ss}_1) \# (X_2 \times S^1, \tilde{\ss}_2^\tau))$ is $\eta^2$ in a fiber of $TF$ smash a Lie framed circle in $\S^{\R}$. Hence for $n\neq 13$, any $n$-cell in $S^{\R}\wedge TF$ is mapped to $S^{10}$ trivially by $BF^{\{e\}} ((X_1 \times S^1, \tilde{\ss}_1) \# (X_2 \times S^1, \tilde{\ss}_2^\tau))$.
\end{proof}

\section{Equivariant family Bauer-Furuta invariant of the connected sum of two tori}

To address the dimension issue in the previous section, we can consider the $S^1$-equivariant Bauer-Furuta invariant. By the equivariant Hopf theorem, we can convert 
\[
BF^{\{S^1\}} ((X_1 \times S^1, \tilde{\ss}_1) \# (X_2 \times S^1, \tilde{\ss}_2^\tau)) \in \{S^{\R}\wedge TF , S^{2\mathbb{H} + 6\R} \}^{S^1}
\] to a nonequivariant stable mapping class, if $\{S^{\R}\wedge TF , S^{2\mathbb{H} + 6\R} \}^{S^1}$ has no fixed points. However, the base of the bundle $F$ is the Picard torus $\mathcal{T}^8$ fixed by the $S^1$-action. 

To address this issue we have to use a refinement of the Bauer-Furuta invariant (see Definition \ref{def:quotient_inv}):

\begin{definition}
For a spin manifold $X$, define its quotient Bauer-Furuta invariant of the $\text{Spin}^c$-structure $\ss$ to be:
\[
\text{BF}_\text{quotient}^{\text{Pin}(2)} (X, \ss) \in \{ TF_0/ Pic(X),  S^{n\mathbb{H}+b^+_2(X)\tilde{\R}} \}^{\text{Pin}(2)},
\]
where $TF_0$ is the Thom space of a rank $m$ quarternion bundle over $Pic(X) = T^{b_1(X)}$, such that 
\[
m-n = \frac{\sigma(X)}{4}.
\]
\end{definition}

For family invariant we can similarly define an invariant with domain modified. For example:
\begin{definition}
Define the quotient $S^1$-equivariant Bauer-Furuta invariant  of the Dehn twist on a sum of two homology tori to be:
\[
BF^{S^1}_\text{quotient} ((X_1 \times S^1, \tilde{\ss}_1) \# (X_2 \times S^1, \tilde{\ss}_2^\tau)) \in \{S^{\R}\wedge (TF / \mathcal{T}^8) , S^{\mathbb{H} + 6\R} \}^{S^1}.
\]
\end{definition}

These invariants work as well as the ordinary BF invariant, because in the Seiberg-Witten equation, the Picard torus is always mapped to zero (while the kernel of the index bundle might be mapped to nonzero self-dual $2$-form).

\subsection{Computation of the quotient $S^1$-equivariant Bauer-Furuta invariant}

Now consider the domain of this invariant. Let $S(F)$ be the sphere bundle of $F$. The fiber of $S(F)$ is $S(\mathbb{H})$. The structure map and the $S^1$-action on $S(F)$ are induced by those on $F$. Notice that the quotient space of $S(F)$ under the $S^1$-action is a bundle $S(F)/S^1$ with fiber $S^2$. Let $G$ be an $\R^3$-bundle whose sphere bundle is $S(G) =S(F)/S^1$. The fiber of $G$ is $\mathbb{H}/S^1$. We will prove

\begin{align}
\{S^{\R}\wedge (TF / \mathcal{T}^8) , S^{\mathbb{H} + 6\R} \}^{S^1} &= 
\{S^{\R}\wedge (TF /( \mathcal{T}^8\cup \{*\})) , S^{\mathbb{H} + 6\R} \}^{S^1}\label{equ1}\\
&= 
\{S^{2\R} \wedge( S(F)_+) , S^{\mathbb{H} + 6\R} \}^{S^1}\label{equ2}\\
&= \{ S^{2\R} \wedge( S(G)_+) , S^{ 4\R}\wedge S^{6\R} \}\label{equ5} \\
&= \{ S^{2\R} \wedge( S(G)_+) , S^{ 10\R} \} \\
&= \{ S^{2\R} \wedge S(G) , S^{ 10\R} \} \label{equ6}
\end{align}

Proof of equation (\ref{equ1}): Note that $TF / \mathcal{T}^8$ has two distinguished points: the infinity point obtained from the construction of the Thom space, and the point obtained from collapsing the torus $\mathcal{T}^8$. Let $S^0 \to TF / \mathcal{T}^8$ be the inclusion of these two points. Then we have the following cofiber sequence (in the stable homotopy category):
\[
\cdots \to S^0 \to TF / \mathcal{T}^8 \to TF /( \mathcal{T}^8\cup \{*\})) \to \Sigma S^0 = S^1 \to \cdots
\]
and therefore 
\[
\cdots \leftarrow \{S^0, S^{\mathbb{H} + 5\R} \}^{S^1} \leftarrow \{TF / \mathcal{T}^8 , S^{\mathbb{H} + 5\R} \}^{S^1} \leftarrow \{TF /( \mathcal{T}^8\cup \{*\})) , S^{\mathbb{H} + 5\R} \}^{S^1}  \leftarrow \{S^1, S^{\mathbb{H} + 5\R} \}^{S^1} \leftarrow \cdots
\]
is exact. By the equivariant Hopf theorem and the dimension reason, 
\[
\{S^0, S^{\mathbb{H} + 5\R} \}^{S^1} = \{S^1, S^{\mathbb{H} + 5\R} \}^{S^1} = \{1\}.
\]
Hence 
\begin{align*}
\{S^{\R}\wedge (TF / \mathcal{T}^8) , S^{\mathbb{H} + 6\R} \}^{S^1} &= \{ TF / \mathcal{T}^8 , S^{\mathbb{H} + 5\R} \}^{S^1}\\
&= \{TF /( \mathcal{T}^8\cup \{*\}) , S^{\mathbb{H} + 5\R} \}^{S^1} \\
&= \{S^{\R}\wedge (TF /( \mathcal{T}^8\cup \{*\})) , S^{\mathbb{H} + 6\R} \}^{S^1}.
\end{align*}

Proof of equation (\ref{equ2}): To form $TF /( \mathcal{T}^8\cup \{*\}) $, we first take the disk bundle $D(F)$ and then collapse its boundary and its zero section. This is equivalent to collapse the boundary of $I \times  S(F)$. To form the latter, we first take the product $S^{\R} \times  S(F)$ and then collapse $\{1\}\times  S(F)$. Hence we have (imagine a low dimension analogue where $F$ is a rank $1$ bundle over a circle)
\[
TF /( \mathcal{T}^8\cup \{*\}) = S^{\R}\wedge S(F)_+.
\]
Hence we have 
\[
\{S^{\R}\wedge (TF /( \mathcal{T}^8\cup \{*\})) , S^{\mathbb{H} + 6\R} \}^{S^1} = \{S^{2\R} \wedge( S(F)_+), S^{\mathbb{H} + 6\R}\}^{S^1}.
\]

Proof of equation (\ref{equ5}): 
\[
\{S^{2\R} \wedge( S(F)_+) , S^{\mathbb{H} + 6\R} \}^{S^1}
= \{ S^{2\R} \wedge( S(G)_+) , S^{ 4\R}\wedge S^{6\R} \}.
\]
For any subgroup $H$ of a compact Lie group $G$, we have (see the Handbook of homotopy theory, Proposition 17.3.17)
\[
\{G_+\wedge_H X,Y\}^G \cong \{ X, i_H^* Y\}^H
\]
where $X$ is an object in the $H$-Spanier-Whitehead category and $Y$ is an object in the $G$-Spanier-Whitehead category, and $i^*_H$ is the forgetting functor. In our case we just have $G= S^1$, $H = \{e\}$, $X =S(G)_+$, $Y=S^{\mathbb{H} +6\R}$, and $i_{\{e\}}^* Y=S^{4\R +6\R}$. Notice that $G_+\wedge X$ is actually $S^1_+\wedge S(G)_+ = (S^1 \times S(G))_+$, where $S^1 \times S(G)$ on each fiber is a trivial fibration $S^1 \times S^2$. What we want is $S(F)$, which on each fiber is a Hopf fibration $S^3$. But the proof is similar, based on the condition that the action on the domain is free away from the base point.

Proof of Equality (\ref{equ6}): Because of the dimension reason, we can use a cofiber sequence similar to the first one.

We have prove that
\begin{theorem}
We can convert the quotient $S^1$-equivariant Bauer-Furuta invariant to a nonequivariant element in
\[
\{S^{\R}\wedge (TF / \mathcal{T}^8) , S^{\mathbb{H} + 6\R} \}^{S^1} \cong \{ S^{2\R} \wedge S(G) , S^{ 10\R}\}. 
\]
\end{theorem}
Now it's easy to analyze the CW structure of the domain $S(G)$. We compute the dimension of the cells:
\begin{align*}
    \mathcal{T}^8&\text{              }& S(\mathbb{H})/S^1 \\
   0 && 0 \text{              }\\
  1&&2\text{              }\\
   2&&\\
     3&&\\
      4&&\\
       5&&\\
        6&&\\
         7&&\\
         8&&\\
\end{align*}
Again since the dimension of the target is $10$, we only need to consider $9$-, $10$-, $11$-, and $12$-cells of the domain $S^{2\R} \wedge S(G) $. They come from the $7$-, $8$-, $9$-, and $10$-cells of $S(G)$. The structure of $G$ is induced from the structure of $F$, so all Stiefel–Whitney classes vanish, and therefore all Steenrod squares vanish. The Hopf elements $\eta$ is detected by the Steenrod square, hence the only possible nontrivial attaching maps are the ones from the $10$-cell to $7$-cells in $S(G)$. 

In general such attaching maps can be $\eta^2$, but in this case, those cells are obtained from the bundle $G$ and the attaching maps come from the structure of $G$. The top cell of $S(G)$ is the product of the top cell (we denote it by $\sigma^8$) of $\mathcal{T}^8$ and the $2$-cell of the fiber $S(\mathbb H) / S^1 = S^2$, while $7$-cells of $S(G)$ correspond to products of $5$-cells of $\mathcal{T}^8$ and the $2$-cell of the fiber $S(\mathbb H) / S^1 = S^2$. 

Let $\sigma^5 \times D^2$ be any one of the $7$-cells of $TG$. Consider the attaching map from $\partial  (\sigma^8 \times D^2)$ to $\sigma^5 \times D^2$. Pick any generic point $x$ on $\sigma^5$. The unit sphere of the normal bundle at $x\in \sigma^5 \subset \mathcal{T}^8$ is a $2$-sphere, but $\pi_2(SO(3))$ is trivial. This means that for any point $y$ in the fiber over $x$, the preimage of $y$ under the attaching map is a trivial framed sphere. Hence the attaching map is trivial.

Now each cell in $(S^{2\R} \wedge S(G)  )/ (S^{2\R} \wedge S(G) )^{(8)}$ has trivial attaching map stably. By the Atiyah-Hirzebruch spectral sequence,
\[
\{ S^{2\R} \wedge S(G)  , S^{ 10\R} \} \cong  \binom{8}{0}\pi^{10}(S^{12}) \oplus\binom{8}{1}\pi^{10}(S^{11}) \oplus\binom{8}{2}\pi^{10}(S^{10}) .
\]

From the observation of Kronheimer-Mrowka~\cite{KM20}, the preimage of a generic point under 
\[
BF^{\{e\}} ((X_1 \times S^1, \tilde{\ss}_1) \# (X_2 \times S^1, \tilde{\ss}_2^\tau))
\]
is a product between $\eta \times \eta'$ in a fiber of $TF$ and a Lie framed circle $\eta''$ in $\S^{\R}$. The group $S^1 $ acts on the torus $\eta \times \eta'$, and $(\eta \times \eta')/S^1$ is a Lie framed circle in a fiber of $TG$. We conclude that 
\[
BF^{\text{Pin}(2)}_\text{quotient} ((X_1 \times S^1, \tilde{\ss}_1) \# (X_2 \times S^1, \tilde{\ss}_2^\tau)) =( (\eta \times \eta')/S^1 )\times \eta''
\]
is the generator of $\pi^{10}(S^{12}) \subset \{ S^{\R} \wedge  TG , S^{ 10\R} \}$.

Hence by Proposition \ref{prop:free-invariant-detect-pi-0-Diff}, the main theorem \ref{main-thm} follows.

\bibliographystyle{alpha}
%\bibliographystyle{plain}
%\bibliography{./reference/braids_links}
%\bibliography{./reference/diff}
\bibliography{./T4}

\begin{thebibliography}{McC81}

\bibitem[Ada84]{Adams}
J.~F. Adams.
\newblock {P}rerequisites (on equivariant stable homotopy) for {C}arlssons's
  lecture.
\newblock In Ib~H. Madsen and Robert~A. Oliver, editors, {\em Algebraic
  Topology Aarhus 1982}, pages 483--532, Berlin, Heidelberg, 1984. Springer
  Berlin Heidelberg.

\bibitem[Bau04]{Bauer2}
Stefan Bauer.
\newblock A stable cohomotopy refinement of {S}eiberg-{W}itten invariants.
  {II}.
\newblock {\em Invent. Math.}, 155(1):21--40, 2004.

\bibitem[BF02]{BF02}
Stefan~A. Bauer and Mikio Furuta.
\newblock A stable cohomotopy refinement of {S}eiberg-{W}itten invariants: I.
\newblock {\em Inventiones mathematicae}, 155:1--19, 2002.

\bibitem[BK22]{BK_2022_BF}
David Baraglia and Hokuto Konno.
\newblock On the {B}auer–{F}uruta and {S}eiberg–{W}itten invariants of
  families of 4‐manifolds.
\newblock {\em Journal of Topology}, 15(2):505–586, 05 2022.

\bibitem[KM20]{KM20}
Peter Kronheimer and Tomasz Mrowka.
\newblock The {D}ehn twist on a sum of two ${K}3$ surfaces, 01 2020.

\bibitem[Lin23]{Lin23}
Jianfeng Lin.
\newblock Isotopy of the {D}ehn twist on ${K}3 \# {K}3$ after a single
  stabilization.
\newblock {\em Geometry \& Topology}, 27:1987--2012, 07 2023.

\bibitem[McC81]{McCULLOUGH81}
Darryl McCullough.
\newblock Connected sums of aspherical manifolds.
\newblock {\em Indiana University Mathematics Journal}, 30(1):17--28, 1981.

\bibitem[RS00]{RS00}
Daniel Ruberman and Sašo Strle.
\newblock Mod 2 {S}eiberg-{W}itten invariants of homology tori.
\newblock {\em Mathematical Research Letters}, 7, 04 2000.

\end{thebibliography}
\end{document}